\documentclass[12pt,a4paper]{amsart}
\usepackage{amsfonts}
\usepackage{amsthm}
\usepackage{amsmath}
\usepackage{amssymb}
\usepackage{amscd}
\usepackage{t1enc}
\usepackage[mathscr]{eucal}
\usepackage{indentfirst}
\usepackage{graphicx}
\usepackage{graphics}
\numberwithin{equation}{section}

\begin{document}

\title{The density Tur\'an problem}

\author[P. Csikv\'ari]{P\'{e}ter Csikv\'{a}ri}
\address{E\"{o}tv\"{o}s Lor\'{a}nd University \\ Department of Computer Science
 \\ H-1117 Budapest
 \\ P\'{a}zm\'{a}ny P\'{e}ter s\'{e}t\'{a}ny 1/C \\ Hungary \& Alfr\'ed R\'enyi
 Institute of 
 Mathematics \\ H-1053 Budapest \\ Re\'altanoda u. 13-15. \\ Hungary}
\email{csiki@cs.elte.hu}

\author[Z. L. Nagy]{Zolt\'an L\'or\'ant Nagy}

\address{E\"{o}tv\"{o}s Lor\'{a}nd University \\ Department of Computer Science
 \\ H-1117 Budapest
 \\ P\'{a}zm\'{a}ny P\'{e}ter s\'{e}t\'{a}ny 1/C \\ Hungary \& Alfr\'ed R\'enyi
 Institute of 
 Mathematics \\ H-1053 Budapest \\ Re\'altanoda u. 13-15. \\ Hungary}
\email{nagyzoltanlorant@gmail.com}

\thanks{Corresponding author: Zolt\'an L\'or\'ant Nagy}

\thanks{The research  was partially supported by the
  Hungarian National Foundation for Scientific Research (OTKA), Grant no. 
K67676 and K81310}



\date{}

\addtolength{\textheight}{2cm}

\theoremstyle{plain}
\newtheorem{theorem}{\bf Theorem}[section]
\newtheorem{lemma}[theorem]{\bf Lemma}
\newtheorem{cor}[theorem]{\bf Corollary}
\newtheorem{prop}[theorem]{\bf Proposition}
\newtheorem{conj}[theorem]{\bf Conjecture}
\newtheorem{claim}[theorem]{\bf Claim}
\newtheorem{construction}[theorem]{\bf Construction}
\newtheorem{Al}[theorem]{\bf Algorithm}
\newtheorem{counte}[theorem]{\bf Counterexample}

\theoremstyle{definition}
\newtheorem{example}[theorem]{\bf Example}
\newtheorem{remark}[theorem]{\bf Remark}
\newtheorem{defn}[theorem]{\bf Definition}
\newtheorem{problem}[theorem]{\bf Problem}
\newtheorem{obs}[theorem]{\bf Observation}
\newtheorem{nota}[theorem]{\bf Notation}

\begin{abstract} Let $H$ be a graph on $n$ vertices and let the blow-up graph
  $G[H]$ be defined as follows. We replace each vertex $v_i$ of $H$ by a cluster
  $A_i$ and connect some pairs of vertices of $A_i$ and $A_j$ if $(v_i,v_j)$ was
  an edge of the graph $H$.  As usual, we
define the edge density between $A_i$ and $A_j$ as
$d(A_i,A_j)=\frac{e(A_i,A_j)}{|A_i||A_j|}.$
We study the following problem. Given densities $\gamma_{ij}$  for each edge
$(i,j)\in E(H)$. Then one has
to decide whether there exists a blow-up graph $G[H]$ with edge densities at
least  $\gamma_{ij}$ such that one cannot choose a vertex from each cluster so
that the obtained graph is isomorphic to $H$, i.e, no $H$ appears as a
transversal in $G[H]$.  We call $d_{crit}(H)$ the
maximal value for which there exists a blow-up graph $G[H]$ with edge
densities $d(A_i,A_j)=d_{crit}(H)$ $((v_i,v_j)\in E(H))$ not containing $H$ in the
above sense. Our main goal is to determine the critical edge density and to
characterize the extremal graphs.

First in the case of tree $T$ we give an efficient algorithm to decide
whether a  given set of edge densities ensures the existence of a transversal
$T$ in the blown up graph. Then we give general bounds on $d(H)$ in terms of
the maximal degree. In connection  with the extremal structure, the so-called
star  decomposition  is proven to give the best construction for
$H$-transversal-free blow-up graphs for several graph classes. 

Our approach applies algebraic graph theoretical, combinatorial and
probabilistic tools. 
\end{abstract}

\maketitle

\section{Introduction}

Given a simple, connected  graph $H$, we define a blow-up graph $G[H]$ of $H$
as follows.  Replace each
vertex $v_i\in V(H)$ by a cluster $A_i$ and connect vertices between the
clusters $A_i$ and $A_j$ (not necessarily all) if $v_i$ and $v_j$ were adjacent
in $H$.  As usual, we
define the density between $A_i$ and $A_j$ as
$$d(A_i,A_j)=\frac{e(A_i,A_j)}{|A_i||A_j|},$$
where $e(A_i,A_j)$ denotes the number of edges between the clusters $A_i$ and
$A_j$. We say that the graph $H$ is a transversal of $G[H]$ if $H$ is a
subgraph of $G[H]$ such that we have a homomorphism $\varphi:
V(H)\rightarrow V(G[H])$ for which $\varphi(v_i)\in A_i$ for all
$v_i\in V(H)$.  We will
also use the terminology that $H$ is a factor of $G[H]$. 

The density Tur\'an problem asks to determine the critical edge density
$d_{crit}$ which ensures the existence of the subgraph $H$ of $G[H]$ as a
transversal. What does it mean? Assume that for
all $e=(v_i,v_j) \in E(H)$ we have $d(A_i,A_j)>d_{crit}$ then no matter how the
graph $G[H]$ looks like, it induces the graph $H$ as a transversal. On
the other hand, for any $d<d_{crit}(H)$ there exists a blow-up graph $G[H]$
such that $d(A_i,A_j)>d$ for all $(v_i,v_j)\in E(H)$ and it does not contain $H$
as a transversal. Clearly, the critical edge density of the graph $H$ is the
largest one of the critical edge densities of its components. Thus we 
 will assume that $H$ is a connected graph throughout the paper.

The  problem in view was studied in  \cite{nagy1}. A very closely related
variant of this problem was mentioned in the book Extremal Graph Theory
\cite{bol1} on page 324.  There are many papers where density condition is
replaced by minimal degree constraint \cite{bol2,bol3,jin,yuster}.
\bigskip

It will turn out that it is useful to consider the following more general
problem. Assume that a density $\gamma_e$ is given for every edge $e\in
E(H)$. Now the problem is to decide whether the densities $\{\gamma_e\}$
ensure the existence of the subgraph $H$ as a transversal or one can construct
a blow-up graph $G[H]$ such that $d(A_i,A_j)\geq \gamma_{ij}$, yet the graph
$H$ does not appear in $G[H]$ as a transversal. This more general approach
allows us to use inductive proofs.
We refer to this general setting as inhomogeneous condition on the edge
densities while the above condition of having a common lower bound
$d_{crit}(H)$ for the densities is called the homogeneous case. 

It will turn out that an even more general setting is worth considering,
that is, to consider weighted blow-up graphs (see Section 2). 

The paper is organized as follows. We end this introductory section by setting
down the notations. In Section 2 we introduce the most important concepts
through an example and  we sketch the main results of
\cite{nagy1} and some useful lemmas. Section 3 is devoted to the case when $H$
is a tree. This case is covered in \cite{nagy1} in the homogeneous case, showing
that $$d_{crit}(T)=1-\frac{1}{\lambda^2_{max}(T)},$$ where $\lambda_{max}(T)$
denotes the maximal 
eigenvalue of the adjacency matrix of the tree. In the inhomogeneous case a
set of edge densities is given for a blown up graph $T_n$, and we have to
decide whether the edge densities ensure the existence of a factor $T_n$. 
We give an efficient algorithm to do it. The proof is based on the strong
connection with the multivariate matching polynomial. 
In Section 4 by the application of the Lov\'asz local lemma and its extension,
it is shown that $$d_{crit}(H)<1-\frac{1}{4(\Delta(H)-1)},$$ where $\Delta(H)$
is the maximal degree of $H$. 
The extremal structures are investigated in Section 5. Here we give a recursive
construction for blown up graphs not containing the corresponding transversal,
and examine for which classes of graphs it gives the extremal structure. These
constructions also give lower bounds for the critical edge density in the
homogeneous case. 
\bigskip

\centerline{$\star$\ \ \ \ $\star$\ \ \ \ $\star$}
\bigskip

Throughout the paper, we use the following notation.

\begin{nota}

$H=(V(H), E(H))$ will be a connected graph on the labeled vertices $\{ 1,\dots
,n\}$.  

$G[H]$ denotes a blown up graph of $H$ on $n$ clusters, where the cluster
$A_i$ corresponds to the vertex $i\in V(H)$. If all densities equal $1$ in
$G[H]$ then we call it a complete blown up graph of $H$. 

Graph $S_n, P_n, C_n$ denote the star, the path and the cycle on $n$
vertices, respectively. As usual, $K_n$ and  $K_{m,n}$ denote the complete,
and the complete bipartite graphs. $T_n$ denotes an arbitrary tree on $n$
vertices. 

$d_{crit}(H)$ is the critical edge density assigned to $H$, while $d_e$ is the
edge density between $A_i$ and $A_j$ if $e=ij \in E(H)$. 

$\Delta (H)$ will denote the maximum degree in $H$, $D_i$ the degree of vertex
$i$, while $N(z)$ will denote the neighborhood of vertex $z$.   
\end{nota}

Now we define the weighted version of the well-known independence and matching
polynomials. 

\begin{nota}

Let $G$ be a graph and assume that a positive
  weight function $w:V(G)\rightarrow \mathbb{R^+}$ is given. Then let 
$$I((G,\underline{w});t)=\sum_{S\in \mathcal{I}}\left( \prod_{u\in
    S}w_u\right)(-t)^{|S|},$$ 
where the summation goes over the set $\mathcal{I}$ of all independent sets $S$
of the graph $G$ including the empty set. When $\underline{w}=\underline{1}$
we simply write $I(G,t)$ instead of $I((G,\underline{1});t)$ and we call
$I(G,t)$ the independence polynomial of $G$.  Clearly,
$$I(G,t)=\sum_{k=1}^ni_k(G)(-1)^kt^k,$$
where $i_k(G)$ denotes the number of independent sets of size $k$ in the graph
$G$.

Let $G$ be a graph and assume that a positive
  weight function $w:E(G)\rightarrow \mathbb{R^+}$ is given. Then let 
$$M((G,\underline{w});t)=\sum_{S\in \mathcal{M}}\left( \prod_{e\in
    S}w_e\right)(-1)^{|S|}t^{n-2|S|},$$ 
where the summation goes over the set $\mathcal{M}$ of all independent edge
sets $S$ of the graph $G$ including the empty set. In the case when
$\underline{w}=\underline{1}$ we call the polynomial
$$M(G,t)=M((G,\underline{1});t)=\sum_{k=0}^{n/2}(-1)^km_k(G)t^{n-2k}$$
the matching polynomial of $G$, where $m_k(G)$ denotes the number of $k$
independent edges (i.e., the $k$-matchings) in the graph $G$.

A closely related variant of the weighted matching polynomial is the
multivariate matching polynomial defined as follows. Let $x_e$'s be variables
assigned to each edge of a graph. The \textit{multivariate matching
  polynomial}  $F$ is defined as follows: 
$$F(\underline{x_e},t)=\sum_{M\in \mathcal{M}}(\prod_{e\in M}x_e)(-t)^{|M|},$$
where the summation goes over the matchings of the graph including the empty
matching.

Clearly, if $L_G$ denotes the line graph of the graph $G$ we have 
$$F(\underline{x_e},t)=I((L_G,\underline{x_e});t)$$
or in other words,
$$t^nF(\underline{x_e},\frac{1}{t^2})=M((G,\underline{x_e});t).$$
\end{nota}

\section{Preliminaries} 

In this section we motivate some key definitions through an example graph.\\
The diamond is the unique simple graph on $4$ vertices and $5$ edges,
generally denoted by $K_4^{-}$. 

\begin{figure}[ht]
\begin{minipage}[b]{0.25\linewidth}
\begin{center}
\scalebox{.40}{\includegraphics{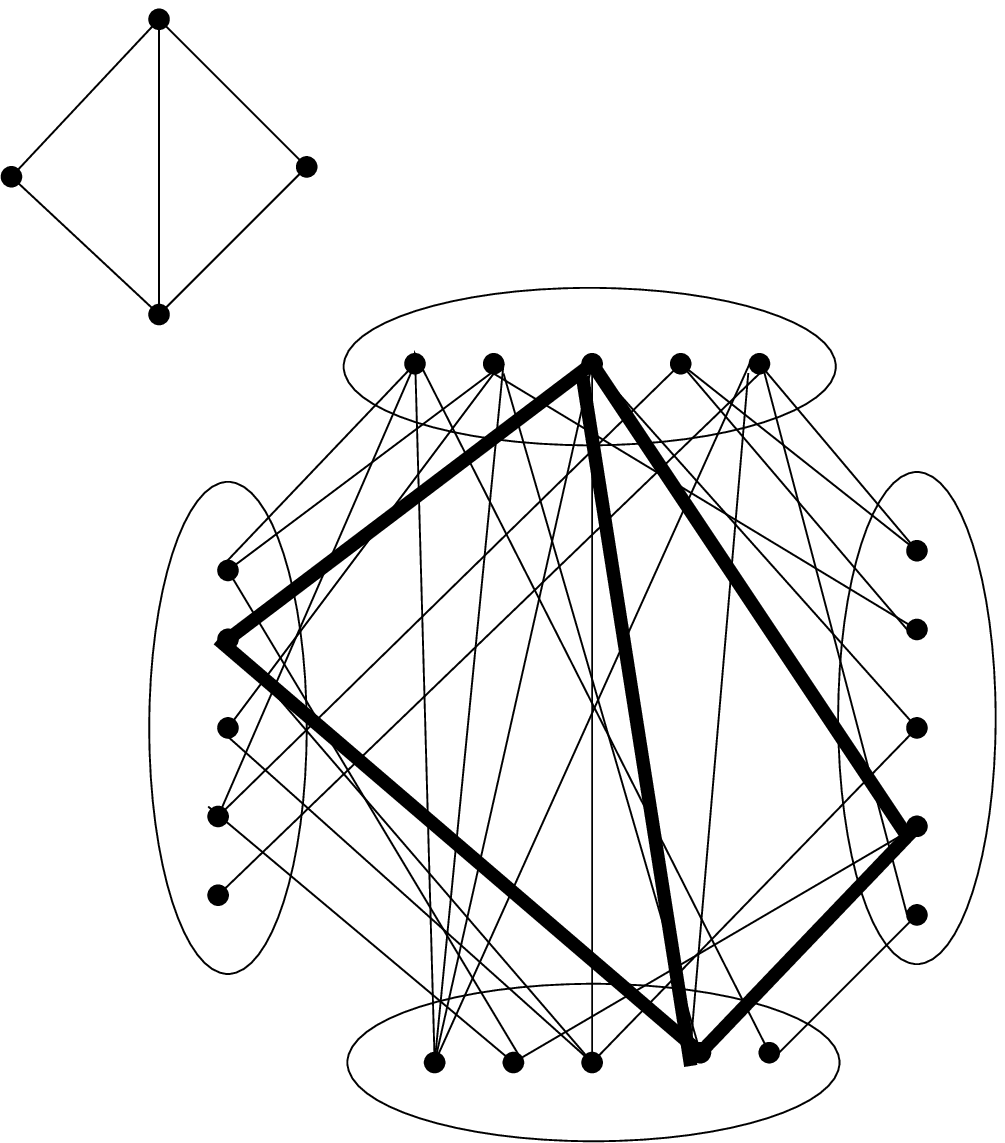}}  
\end{center}
\end{minipage}
\hspace{1cm}
\begin{minipage}[b]{0.25\linewidth}
\begin{center}
\scalebox{.40}{\includegraphics{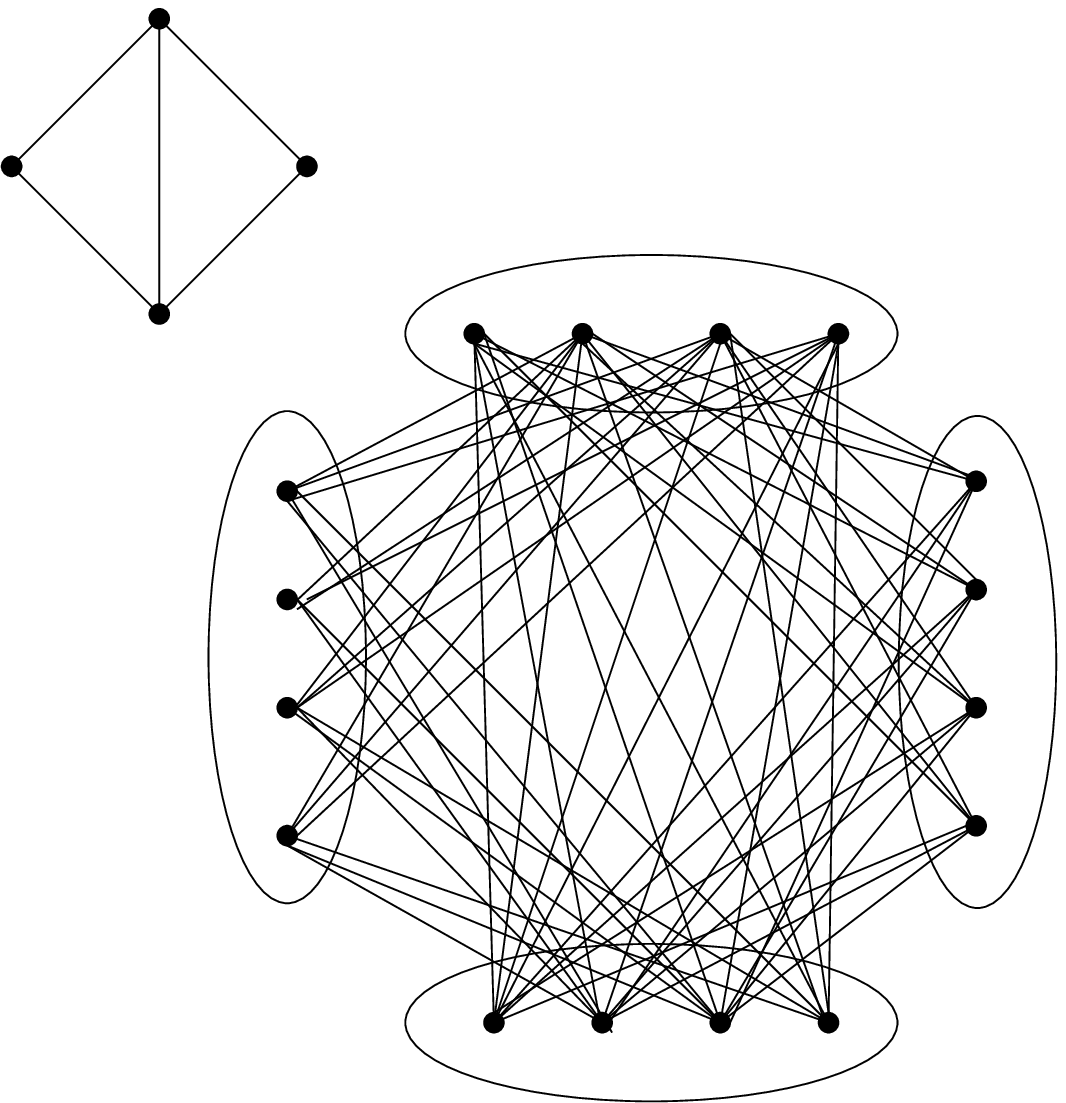}} 
\end{center}  
\end{minipage}
\hspace{1cm}
\begin{minipage}[b]{0.25\linewidth}
\begin{center}
\scalebox{.40}{\includegraphics{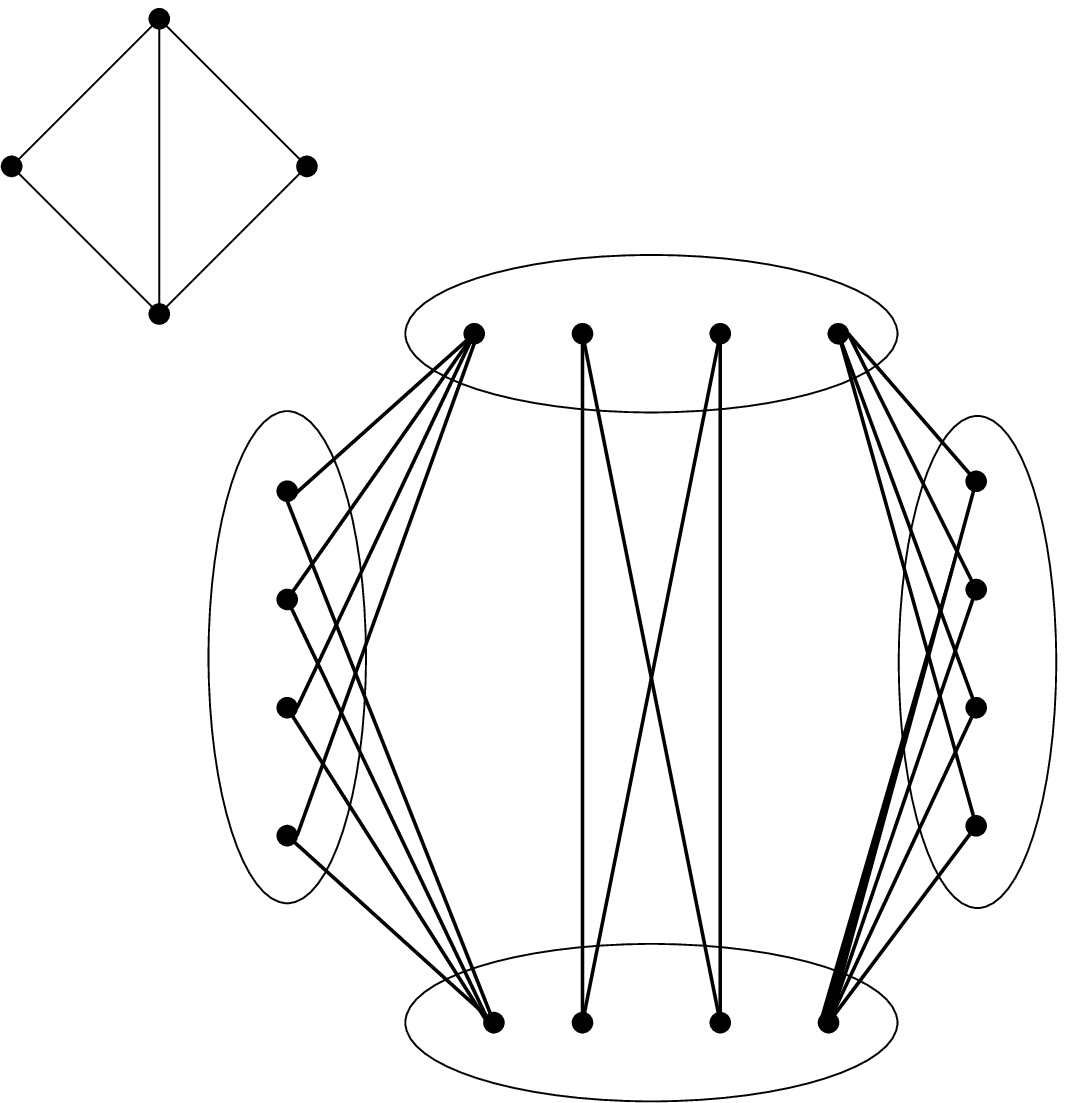}}  
\end{center}
\end{minipage}

\caption{Blow-up graphs of diamonds.} 

\end{figure}

\bigskip

In the  figure above, the first blow-up graph of the diamond contains the
diamond as a transversal. The second blow-up graph does not contain the
diamond as a transversal although the edge density is $3/4$ between any two
clusters. To see it, we gave the \textit{complement of the blow-up
  graph with respect to   the complete blow-up graph}; in what follows we
will simply call this graph the \textit{complement graph} and we will denote
it by $\overline{G[H]|H}$. In the ``complement language'' the claim is the 
following: if one chooses one vertex from each cluster then we cannot avoid
choosing both ends of a complementary edge. This is true indeed: whichever
vertex we  choose from the ``right'' and ``left'' clusters we cannot choose
the rightmost and leftmost vertices of the upmost and downmost clusters; so we
have to choose a vertex from the middle of these clusters, but they are all
connected by complementary edges. 

We also see that this construction was a bit redundant in the sense that each
vertex from the right and left clusters had the same role. This motivates the
following definition.

\begin{figure}[h!]
\begin{center}
\scalebox{.55}{\includegraphics{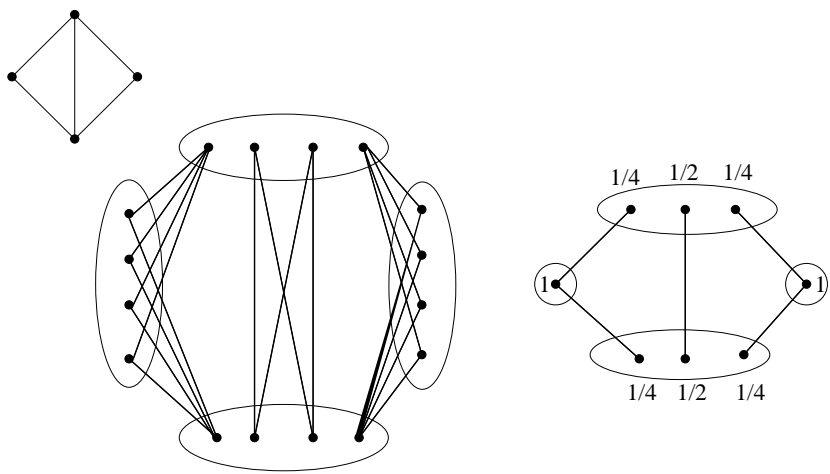}}   \caption{Weighted
  blow-up graph.} 
\end{center}
\end{figure} 
    
\begin{defn} A weighted blow-up graph is a blow-up graph where a non-negative
  weight $w(u)$ is assigned to each vertex $u$ such that the total weight of
  each cluster is $1$. The density  between clusters $A_i$ and $A_j$ is
$$d_{ij}=\sum_{(u,v)\in E \atop u\in A_i,v\in A_j}w(u)w(v).$$
\end{defn}

This definition also has the advantage that now we can allow irrational
weights as well. (But this does not change the problem since we can
approximate any irrational weight by rational weights and then we blow up the
construction with the common denominator of the weights.) The following
result of the second author  \cite{nagy1}  also shows that the problem in this
framework is much more convenient. Note that this result is a simple
generalisation of a statement of Bondy, Shen, Thomass\'e and Thomassen
\cite{bond}.

\begin{theorem} {\rm \cite{nagy1}} \label{ZolTh1} If
  there is a construction of a 
  blow-up graph $G[H]$ not containing $H$ then there is a construction of a
  weighted blow-up graph $G'[H]$ not containing $H$, where
\begin{itemize} 
\item each edge density is at least as large as in $G[H]$,
\item the cluster $V_i$ contains at most as many vertices as the degree of the
  vertex $v_i$ in the graph $H$.
\end{itemize}
\end{theorem}

The importance of this theorem lies in the fact that  if we are looking for
the critical edge density we only have to check 
those constructions where each cluster contains a bounded number of
vertices. So in fact, we have to check a finite number of configurations and
we only have to decide that which configuration has a weighting providing the
greatest density. In general, the number of possible configurations is very
large, still it has some notable consequences. For instance, there is a ``best''
construction in the sense that if we have construction for
${\gamma_e-\varepsilon}$ for every $\varepsilon$ then we have a construction
with edge densities $\gamma_e$. Indeed, we have a compact space (finite number
of configurations) and the edge densities are continuous functions of the
weights.

With a small  extra idea one can prove the following  important corollary of
this theorem. 

\begin{theorem} {\rm \cite{nagy1}} There is a weighted
  blow-up graph $G[H]$ not containing $H$  where each edge density is exactly
  the critical edge density. 
\end{theorem}

From this theorem one can deduce the following ones.

\begin{prop} {\rm \cite{nagy1}}\label{sub} If $H_1$ is a subgraph of
  $H_2$ then for the critical edge densities we have
$$d_{crit}(H_1)\leq d_{crit}(H_2).$$
If $H_2$ is connected and  $H_1$ is a proper subgraph of $H_2$ then the
inequality is strict.
\end{prop}

A general lower and upper bound was also proved in \cite{nagy1}. The lower
bound is the consequence of Proposition \ref{sub} and the fact that
$d_{crit}(S_n)=1-\frac{1}{n-1}$.  

\begin{prop} \label{becsles}
$(1-\frac{1}{\Delta(H)})\leq d_{crit}(H)\leq (1-\frac{1}{\Delta^2(H)})$.
\end{prop} 

The upper bound will be strengthened in Section 4. It was known that
$$d_{crit}(H)<1-\frac{1}{4(\Delta(H)-1)}$$ also holds for trees. It turned out
that it is a general upper bound. 
\medskip

At last, let us mention a theorem of Adrian Bondy, Jian Shen, St\'ephan
Thomass\'e and Carsten Thomassen. On the one hand, it solves the inhomogeneous
problem for $H=K_3$. On the other hand, it provides a base in some forthcoming
proofs.

\begin{lemma} {\rm \cite{bond}}\label{triangle1} 
 Let $\alpha, \beta,\gamma$ be the edge densities between the clusters of a
 blow-up graph of the triangle. 
If 
$$\alpha \beta+\gamma>1,\ \beta\gamma+\alpha>1,\ \gamma\alpha+\beta>1,$$
then the blow-up graph contains a triangle as a transversal. Otherwise there
exists a weighted blow-up graph with the prescribed edge densities without
containing a triangle. 
\end{lemma}

%%%%%%%%%%%%%%%%%%%%%%%%%%%%%%%%%%%%%%%%%%%%%%%%%%%%%%%%%%%%%%%%%%%%%%%%%%%%%%%%%%%%%%%%%%%%%%%%%%%%%%%%%%%%%%%%%%%%% 
\section{Inhomogeneous case: trees }

%%%%%%%%%%%%%%%%%%%%%%%%%%%%%%%%%%%%%%%%%%%%%%%%%%%%%%%%%%%%%%%%%%%%%%%%%%%%%%%%%%%%%%%%%%%%%%%%%%%%%%%%%%%%%%%%%%%% 

In this section we study the case when the graph $H$ is a tree.

\begin{theorem} \label{TT1} Let $T$ be a tree, $v_n$ is an leaf of
  $T$. Assume that for each edge of $T$ a density $\gamma_e=1-r_e$ is 
  given. Let $T'$ be a tree obtained from $T$ by deleting the leaf $v_n$
  (together with the edge $e_{n-1,n}=v_{n-1}v_n$). Let the densities
  $\gamma'_e$'s be defined as follows: 
$$\gamma'_e=\left\{ {\gamma_e=1-r_e \mbox{ if $e$ is not incident to $v_{n-1}$,}
    \atop     1-\frac{r_e}{1-r_{e_{n-1,n}}} \mbox{ if $e$ is incident
      to $v_{n-1}$.} }\right.$$
Then the set of densities $\gamma_e$ ensure the existence of the factor $T$ if
and only if all $\gamma'_e$'s are between $0$ and $1$ and the set of densities
$\gamma'_e$ ensure the existence of the factor $T'$.
\end{theorem}

\begin{remark} Clearly, this theorem provides us with an efficient algorithm to
  decide whether a given set of densities ensures the existence of a factor (see
  Algorithm~\ref{tree-alg}).
\end{remark}

\begin{proof} First we prove that if all the $\gamma'_e$'s are indeed
  densities and they ensure the existence of the factor  $T'$ then the original
  $\gamma_e$'s ensure the existence of a factor $T$.

Assume that $G[T]$ is a blow-up of $T$ such that the density between $A_i$
and $A_j$ is at least $\gamma_{ij}$, where $A_i$ is the blow-up of the vertex
$v_i$ of $T$. We need to show that it contains a
factor $T$. 

Let us define
$$R=\left\{ v\in A_{n-1}\ |\  \mbox{ $v$ is incident to some edge going between
    $A_{n-1}$ and $A_n$} \right\} .$$
First of all we show that the cardinality of $R$ is large:
$$|R||A_n|\geq e(R,A_n)=\gamma_{n-1,n}|A_{n-1}||A_n|.$$
Thus $|R|\geq \gamma_{n-1,n}|A_{n-1}|$. 

Next we show that many edges are
incident to $R$. Let $v_k$ be adjacent to $v_{n-1}$.  Then we can bound the
number of edges between $R$ and $A_{k}$ as follows:
$$e(R,A_k)\geq e(A_{n-1},A_k)-(|A_{n-1}|-|R|)|A_k|=
|R||A_k|+(\gamma_{k,n-1}-1)|A_k||A_{n-1}|\geq $$
$$\geq |R||A_k|+(\gamma_{k,n-1}-1)\frac{1}{\gamma_{n-1,n}}|R||A_k|=$$
$$=(1-\frac{r_{k-1,n}}{1-r_{n-1,n}})|R||A_k|=\gamma'_{k,n-1}|R||A_k|.$$
Now delete the vertex set $A_n$ and $A_{n-1}\backslash R$ from $G[T]$. Then the
obtained graph is a blow-up of $T'$ with edge densities ensuring the factor
$T'$. But this factor can be extended to a factor $T$ because of the
definition of $R$.

Now we prove that if some $\gamma'_{k,n-1}<0$, then there exists a construction
for a blow-up of $T$ having no factor of $T$.  In fact $\gamma'_{k,n-1}<0$
means that $\gamma_{k,n}+\gamma_{n-1,n}<1$ and so we can conclude that
some construction does not induce the path $u_ku_{n-1}u_n$ where $u_i\in A_i$
 ($i\in \{ k,n-1,n\}$). 

Now assume that all $\gamma'_e$'s are proper densities, but there is a
construction $G'[T']$ with edge-densities at least $\gamma'_e$'s, but which
does not induce a factor $T'$. In this case we can easily construct a blow-up
$G[T]$ of the tree not inducing $T$ by setting $A_{n-1}=R^*\cup A'_{n-1}$ with an
appropriate weight of $R^*=\{v_{n-1}^*\}$ and taking an $A_n=\{v_n\}$ which we
connect to all elements of $A'_{n-1}$, but we do not connect to $v_{n-1}^*$.     
\end{proof}  

\begin{Al} \label{tree-alg} 

\textbf{Step 0.} Given a tree a $T_0$ and edge densities $\gamma^0_e$. Set
$T:=T_0$ and $r_e=1-\gamma^0_e$.\\
\textbf{Step 1.} Consider $(T,\underline{r_e})$. 
\begin{itemize}
\item If $|V(T)|=2$ and $0\leq r_e<1$ then \textbf{STOP}: the 
  densities $\gamma^0_e$ ensure the existence of the transversal $T_0$.
\item If $|V(T)|\geq 2$ and there exists an edge for which $r_e\geq 1$ then 
\textbf{STOP}: the densities  $\gamma^0_e$ do not ensure the existence of the
transversal $T_0$.\\
\end{itemize}
\textbf{Step 2.} If $|V(T)|\geq 3$ and $0\leq r_e<1$ for all edges $e\in E(T)$
then \textbf{do}  pick a vertex $v$ of degree $1$, let $u$ be its unique
  neighbor. Let $T':=T-v$ and 
$$r'_e=\left\{ {r_e \mbox{ if $e$ is not incident to $u$,}
    \atop     \frac{r_e}{1-r_{(u,v)}} \mbox{ if $e$ is incident
      to $u$.} }\right.$$
Jump to Step 1. with $(T,\underline{r_e}):=(T',\underline{r'_e})$.
\end{Al}

In what follows we analyse the above mentioned algorithm. The following
concept will be the key tool.

Let $x_e$'s be variables assigned to each edge of a graph. Recall that we
define the \textit{multivariate matching polynomial}  $F$ as follows:
  $$F(\underline{x_e},t)=\sum_{M\in \mathcal{M}}(\prod_{e\in M}x_e)(-t)^{|M|},$$
where the summation goes over the matchings of the graph including the empty
matching.

The following lemma is a straightforward generalization of the well-known
fact that for trees the matching polynomial and the characteristic polynomial
of the adjacency matrix coincide.
\bigskip

\begin{lemma} \label{char-match}
Let $T$ be a tree on $n$ vertices.
Let us define the following matrix of size $n\times n$. 
The entry $a_{i,j}=0$ if the vertices $v_i$ and $v_j$ are not adjacent and
$a_{i,j}=\sqrt{x_e}$ if $e=v_iv_j\in E(T)$. Let $\phi(\underline{x_e},t)$ be the
characteristic polynomial of this matrix. Then
$$\phi(\underline{x_e},t)=t^nF(\underline{x_e},\frac{1}{t^2})$$
where $F(\underline{x_e},t)$ is the multivariate matching polynomial.
\end{lemma}

\begin{proof} Indeed when we expand the $\det(tI-A)$ we only get non-zero
  terms when the cycle decomposition of the permutation consist of cycles of
  length at most $2$; but these terms correspond to the terms of the matching
  polynomial.  
\end{proof}

\begin{prop} \label{mlr} Let $G$ be a tree and let $t_w(G)$ denote the
  largest real root of the polynomial 
  $M((G,\underline{w});t)$. Let $G_1$ be a subgraph of $G$ then we have 
$$t_w(G_1)\leq t_w(G).$$
\end{prop}

\begin{proof}
This is straightforward after applying Lemma \ref{char-match}
\end{proof}

Note that Proposition~\ref{mlr} holds for arbitrary graph $G$, but we do not
use this stronger version. 

\begin{cor} \label{TC1} Let $T$ be a tree and assume that for each edge $e\in
  E(T)$ a 
  weight $w_e>0$ is assigned. Furthermore, let $T'$ be a subtree of $T$ with
  the induced edge weights. Then the polynomial $F_T(\underline{w_e},t)$ has a
  smaller positive root than the polynomial $F_{T'}(\underline{w_e},t)$.
\end{cor}

\begin{lemma} \label{TL2} Let $T$ be a weighted tree with $\gamma_e=1-tr_e$
  weights. Assume that after running the Algorithm~\ref{tree-alg} we get the
  two node tree with edge weight $0$. Then $t$ is the root of the
  multivariate matching polynomial $F(\underline{r_e},s)$ of the tree $T$. 
\end{lemma}

\begin{proof} We prove the statement by induction on the number of vertices of
  the tree. If the tree consists of two vertices then $0=1-tr_e$ means exactly
  that $t$ is the root of the multivariate matching polynomial of the tree.

Now assume that the statement is true for trees on at most $n-1$ vertices. Let
$T$ be a tree on $n$ vertices and assume that we execute the algorithm for the
pendant edge $e_{n-1,n}=(v_{n-1},v_n)$ in the first step, where the degree of
the vertex $v_n$ is $1$.  Let $T'=T-v_n$. Now we continue executing the
algorithm obtaining the two node tree with edge weight $0$. By induction we get
that $F_{T'}(\underline{r'_e},t)=0$. 

We can expand $F_{T'}$ 
according to whether a monomial contains $x_{k,n-1}$ ($e_{k,n-1}\in E(T')$) or
not. Each monomial can contain at most one of the variables $x_{k,n-1}$ ($v_k\in
N(v_{n-1})$). Thus 
$$F_{T'}(\underline{x_e},s)=Q_0(\underline{x_e},s)-\sum_{v_k\in
  N(v_{n-1})}sx_{k,n-1}Q_k(\underline{x_e},s),$$
where $Q_0$ consists of those terms which contain no $x_{k,n-1}$ and
$-sx_{k,n-1}Q_k$ consists of those terms which contain $x_{k,n-1}$, i.e.,
these terms correspond to the matchings containing the edge $(v_k,v_{n-1})$.
Observe that
$$F_{T}(\underline{x_e},s)=(1-sx_{n-1,n})Q_0(\underline{x_e},s)-\sum_{v_k\in
  N(v_{n-1})}sx_{k,n-1}Q_k(\underline{x_e},s)$$
by the same argument.

Since 
$$0=F_{T'}(\underline{r'_e},t)=Q_0(
\underline{r_e},t)-\sum_{v_k\in
  N(v_{n-1})}\frac{r_{k,n-1}}{1-tr_{n-1,n}}Q_k(\underline{r_e},t)$$
we have 
$$0=(1-tr_{n-1,n})F_{T'}(\underline{r'_e},t)=(1-tr_{n-1,n})Q_0(
\underline{r_e},t)-\sum_{v_k\in
  N(v_{n-1})}r_{k,n-1}Q_k(\underline{r_e},t)=F_{T}(\underline{r_e},t).$$
Hence $t$ is the root of $F_{T}(\underline{r_e},s)$.   
\end{proof}

\begin{theorem} \label{TT2} Let $T$ be a tree and let $\gamma_e=1-r_e$ be edge
  densities. 
Then the edge densities ensure the existence of the tree $T$ as a transversal
if and only if for the multivariate matching polynomial we have
$$F(\underline{r_e},t)>0$$
for all $t\in [0,1]$.
\end{theorem}

\begin{remark} We mention that the really hard part of this theorem is that if 
$$F(\underline{r_e},t)>0$$
for all $t\in [0,1]$ then the edge densities $\gamma_e=1-r_e$ ensure the
existence of the tree $T$ as a transversal. Later we will prove that this is
true for every graph $H$, see Theorem~\ref{C1}.
\end{remark}

\begin{proof} We prove the theorem by induction on the number of vertices. We
  will use Theorem~\ref{TT1}. First we show that if the edge
  densities ensure the existence of the factor $T$
  then $$F(\underline{r_e},t)>0$$ for all $t\in [0,1]$.

Clearly,
$$ F(\underline{r_e},t)=F(\underline{r_et},1).$$
It is also trivial  that the  densities $\gamma_e=1-r_e$ ensure the existence
of a factor $T$ then the densities $\gamma_e=1-tr_e$ $(t\in [0,1])$ ensure the
existence of factor $T$. Hence we only need to prove that if the  densities
$\gamma_e=1-r_e$ ensure the existence of factor $T$ then
$F(\underline{r_e},1)>0$.

We will use the notations of Theorem~\ref{TT1}.
By induction and Theorem~\ref{TT1} we have
$F_{T'}(\underline{r'_e},1)>0$. Now we repeat the argument of Lemma~\ref{TL2}.

As before, we can expand $F_{T'}$ 
according to whether a monomial contains $x_{k,n-1}$ ($e_{k,n-1}\in E(T')$) or
not. Each monomial can contain at most one of the variables $x_{k,n-1}$ ($v_k\in
N(v_{n-1})$). Thus 
$$F_{T'}(\underline{x_e},t)=Q_0(\underline{x_e},t)-\sum_{v_k\in
  N(v_{n-1})}tx_{k,n-1}Q_k(\underline{x_e},t),$$
where $Q_0$ consists of those terms which contain no $x_{k,n-1}$ and
$-tx_{k,n-1}Q_k$ consists of those terms which contain $x_{k,n-1}$, i.e.,
these terms correspond to the matchings containing the edge $(v_k,v_{n-1})$.
We have
$$F_{T}(\underline{x_e},t)=(1-tx_{n-1,n})Q_0(\underline{x_e},t)-\sum_{v_k\in
  N(v_{n-1})}tx_{k,n-1}Q_k(\underline{x_e},t)$$
by the same argument.

Hence  
$$0<F_{T'}(\underline{r'_e},1)=Q_0(
\underline{r_e},1)-\sum_{v_k\in
  N(v_{n-1})}\frac{r_{k,n-1}}{1-r_{n-1,n}}Q_k(\underline{r_e},1).$$ 

So we get that
$$0<(1-r_{n-1,n})F_{T'}(\underline{r'_e},1)=(1-r_{n-1,n})Q_0(
\underline{r_e},1)-\sum_{v_k\in
  N(v_{n-1})}r_{k,n-1}Q_k(\underline{r_e},1)=F_{T}(\underline{r_e},1).$$  
This completes one direction of the proof.
\bigskip

Now we assume that $F(\underline{r_e},t)>0$ for all $t\in [0,1]$. We prove by
contrary that the edge densities $\gamma_e$'s ensure the existence of factor
$T$. Assume that the Algorithm~\ref{tree-alg} stops with 
some $r_{e^{\circ}}\geq 1$. We will call $e^{\circ}$ the violating edge. In the next step we show that for some
$t\in [0,1]$ we can ensure that the algorithm stops
with $r_{e^{\circ}}(t)=1$ when we start with the edge densities
$\gamma_e=1-tr_e$. 

First of all, let us examine what
happens if we decrease the $r_e$'s. If $0<r_e\leq r^*_e$ and $0<r_f\leq r^*_f$
then 
$$\frac{r_e}{1-r_f}\leq \frac{r^*_e}{1-r^*_f}.$$
Hence all $r_i$'s decrease under the algorithm if we decrease $t$.
 
If we set $t=0$ then for the edge densities $\gamma_e=1-tr_e$ the algorithm
gives $1$ for all densities which show up. Since changing $t$ continuously all
densities will change continuously we 
can  choose an appropriate $t\in [0,1]$ for which running our algorithm with
$tr_e$'s instead of $r_e$'s we can assume that the algorithm stops with
$r_{e^{\circ}}(t)= 1$.  

Now consider those vertices and edges
together with the violating edge which were deleted under executing the
algorithm. These edges form a forest. Consider the components of this forest
which contains the violating edge. Let us call this subtree $T_1$. According to
Lemma~\ref{TL2}  our chosen $t$ is the root of the matching polynomial of $T_1$
(clearly, only the deleted edges modified the weight of the violating edge).
On the other hand, we know from Corollary~\ref{TC1} that the matching
polynomial of $T$ has a smaller root than the matching polynomial of $T_1$.
This means that the matching polynomial of $T$ has a root in the interval
$[0,1]$ contradicting the condition of the theorem.  

\end{proof}

\begin{cor} Let $T$ be a tree and assume that all edge densities $\gamma_e$
  satisfy $\gamma_e>1-\frac{1}{\lambda(T)^2}$ where $\lambda(T)$ is the largest
  eigenvalue of the adjacency matrix of $T$. Then $\gamma_e$'s ensure the
  existence of factor $T$. If all $\gamma=1-\frac{1}{\lambda(T)^2}$ then
  there exists a weighted blow-up of $T$ not containing $T$ as a transversal.
In other words,
   $$d_{crit}(T)=1-\frac{1}{\lambda(T)^2}.$$
\end{cor} 

\begin{proof} We can assume that all edge densities are equal to $1-d>
  1-\frac{1}{\lambda^2}$. In this case $dt<\frac{1}{\lambda(T)^2}$ for all $t\in
  [0,1]$ and so
  $$0<\phi_T(\frac{1}{\sqrt{dt}})=(dt)^{-n/2}F_T(\underline{dt},1)=(dt)^{-n/2}F_T(\underline{d},t)$$      
by Lemma~\ref{char-match}.
  By Theorem~\ref{TT2} this implies that the set of edge densities
  $\{\gamma_e\}$ ensure the existence of factor $T$. Theorem~\ref{TT2} also
  implies that  there exists a weighted blow-up with weights
  $\gamma=1-\frac{1}{\lambda(T)^2}$ of $T$ not containing $T$ as a
  transversal.    
\end{proof}

Finally we recall a structure theorem concerning the critical
edge density of trees.  

\begin{prop} {\rm \cite{nagy1}} Let $T$ be a tree. Let us consider the following
  blow-up graph $G[T]$ of $T$. Let the cluster $A_i$ consist of the vertices
  $v_{ij}$ where $j\in N(i)$. If $(i,j)\in E(T)$ then we connect all vertices
  of $A_i$ and $A_j$ except $v_{ij}$ and $v_{ji}$. Then $G[T]$ does not
  contain $T$ as a transversal. 
\end{prop} 

\begin{figure}[h!]
\begin{center}
\scalebox{.35}{\includegraphics{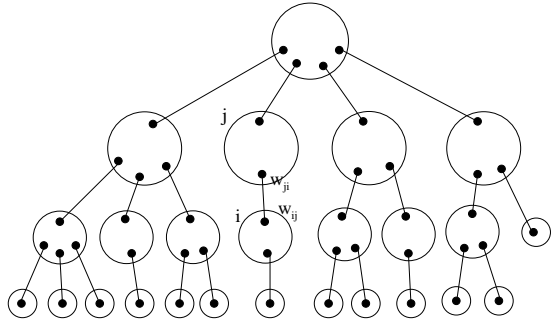}}   \caption{The complement of
  a special blow-up graph of a tree.} 
\end{center}
\end{figure}

\begin{proof} We have to prove that one cannot avoid choosing both endvertices
  of a complementary edge $(v_{ij},v_{ji})$ if one chooses one vertex from
  each cluster. This is indeed true since the set of all vertices of $G[T]$
  can be decomposed to  $(n-1)$ such pairs. Since we have to choose $n$
  vertices we have to choose both vertex from such a pair.
\end{proof}

We show that we can give weights to the vertices of the above constructed
$G[T]$ such that the density will be $1-\frac{1}{\lambda^2}$ where
$\lambda=\lambda(T)$. The following weighting was the idea of Andr\'as G\'acs
\cite{gacs}. 

Recall that there exists a non-negative eigenvector $\underline{x}$ belonging to
the largest eigenvalue $\lambda$ of $T$. So if $v_i$'s are the vertices of $T$
we have 
$$\lambda x_i=\sum_{j\in N(i)}x_j$$
for all $i$. 
Now let us define the weight $w_{ij}$ of the vertex $v_{ij}$ of $G[T]$ as
follows: $w_{ij}=\frac{x_j}{\lambda x_i}\geq 0$. Then we have
$$w(A_i)=\sum_{j\in N(i)}w_{ij}= \sum_{j\in N(i)}\frac{x_j}{\lambda x_i}=1.$$
Furthermore,
$$d(A_i,A_j)=1-w_{ij}w_{ji}=1-\frac{x_j}{\lambda x_i}\frac{x_i}{\lambda
  x_j}=1-\frac{1}{\lambda^2}.$$

%%%%%%%%%%%%%%%%%%%%%%%%%%%%%%%%%%%%%%%%%%%%%%%%%%%%%%%%%%%%%%%%%%%%%%%%%%%%%%%%%%%%%%%%%%%%%%%%%%%%%%%%%%%%%%%%%%%% 

%%%%%%%%%%%%%%%%%%%%%%%%%%%%%%%%%%%%%%%%%%%%%%%%%%%%%%%%%%%%%%%%%%%%%%%%%%%%%%%%%%%%%%%%%%%%%%%%%%%%%%%%%%%%%%%%%%%% 
\section{General bounds } 

Our next aim is to prove good bounds on the critical edge density. Recall that
$(1-\frac{1}{\Delta(H)})\leq d_{crit}(H)\leq (1-\frac{1}{\Delta^2(H)})$ was
known before, see Proposition~\ref{becsles}. 
Our approach is probabilistic. First we give a bound applying the Lov\'asz
local lemma. In fact, we can copy the argument of \cite{alon}. 

\begin{theorem}\label{LLL} {\rm(Lov\'asz local lemma, symmetric case,
    \cite{alon}.)} Let 
  $A_1,A_2,\dots,A_n$ be events in an arbitrary probability space. Suppose
  that each event $A_i$ is mutually independent of all other events, but at
  most $\Delta$ of them. Furthermore, assume that for each $i$,
$$\mathbb{P}(A_i)\leq \frac{1}{e(\Delta+1)},$$
where $e$ is the base of the natural logarithm.
Then
$$\mathbb{P}(\cap_{i=1}^n \overline{A_i})>0.$$
\end{theorem}
   
\begin{theorem}\label{LLLT1} Let $\Delta$ be the largest degree of the graph $H$
  and let $d_{crit}(H)$ be the critical edge density. Then
$$d_{crit}(H)\leq 1-\frac{1}{e(2\Delta-1)},$$
where $e$ is the base of the natural logarithm.
\end{theorem}

\begin{proof} We prove by contradiction. Assume that there exists a blow-up
  graph $G[H]$ of the graph $H$ with edge densities greater than
  $1-\frac{1}{e(2\Delta-1)}$ which does not induce $H$. 

We can assume that all classes of the blow-up graph $G[H]$
contain exactly $N$ vertices. Indeed, we can approximate each weight by a
rational number so that every edge densities are still larger than
$1-\frac{1}{e(2\Delta-1)}$. Then we ``blow up'' the construction by the common
denominator of all weights.

Let us choose a vertex from each class with
equal probability $1/N$ independently of each other. Let $f$ be an edge of
the complement of the graph $G[H]$ with respect to $H$. Let $A_f$ be the event
that we have chosen both endnodes of the edge $f$ (clearly, a bad event we would
like to avoid).  Then $\mathbb{P}(A_f)=1/N^2$ and $A_f$ is independent from
all events $A_{f'}$ where the edge $f'$ has endvertices in different
classes. Thus $A_f$ is independent from all, but at most $(2\Delta-1)rN^2$ bad
events where $r=1-d_{crit}(H)$. Since $r<\frac{1}{e(2\Delta-1)}$ the condition of
Lov\'asz local lemma is satisfied and gives that
$$\mathbb{P}(\cap_{f\in E(\overline{G[H]|H})}A_f)>0.$$
which means that  $G[H]$ induces the graph $H$ (with positive probability)
contradicting the assumption.
\end{proof}

Next, we use a generalization of the Lov\'asz local lemma to improve on the bound
of Theorem \ref{LLLT1}. 
\bigskip

\begin{theorem}{\rm \label{Ath1}[Scott-Sokal] \cite{sokal1} }
Assume that
given a graph $G$ and there is an event  $A_i$ assigned to each node
$i$. Assume that $A_i$ is mutually independent of the events $\{A_k\ |\ (i,k)\in
E(\overline{G})\}$. Set $\mathbb{P}(A_i)=p_i$.
\medskip

\noindent (a)  \textit{Assume that $I((G,\underline{p}),t)>0$ for all $t\in
[0,1]$. Then we have }
$$\mathbb{P}(\cap_{i\in V(G)} \overline{A_i})\geq I((G,\underline{p}),1)>0.$$
\medskip

\noindent (b)  \textit{Assume that $I((G,\underline{p}),t)=0$ for some $t\in
  [0,1]$. 
Then there exists a probability space and a family of events $B_i$ with
$\mathbb{P}(B_i)\geq p_i$ and with dependency graph $G$ such that}
$$\mathbb{P}(\cap_{i\in V(G)} \overline{B_i})=0.$$

\end{theorem}

\begin{theorem} \label{C1} Assume that for the  graph $H$ we have
  $F_H(\underline{r_e},t)>0$ 
  for all $t\in [0,1]$ and some weights $r_e\in [0,1]$ assigned to each
  edge. Then the densities $\gamma_e=1-r_e$ ensure the existence of $H$ as a
  transversal. 
\end{theorem}

\begin{proof} As before, we choose a vertex from each cluster independently
  of each other. We choose the vertex $u$ from the cluster $V_i$ of the graph
  $G[H]$ with probability $w(u)$. We would like to show that we do not choose
  both endvertices  of an edge of the complement $\overline{G[H]|H}$ with
  positive probability.  Let $f=(u_1,u_2)$ be an edge of
the $\overline{G[H]|H}$. Let $A_f$ be the event
that we have chosen both endnodes of the edge $f$ (clearly, a bad event we would
like to avoid). Then $\mathbb{P}(A_f)=w(u_1)w(u_2)$ and $A_f$ is independent
from all events $A_{f'}$ where the edge $f'$ has endvertices in different
classes. Now let us consider the weighted independence polynomial of the graph
determined by the vertices $A_f$ in which we connect $A_f$ and $A_{f'}$ if there
exists a cluster containing endvertices of both $f$ and $f'$. In this graph,
the events $A_f$ where $f$ 
goes between the fixed clusters $V_i,V_j$  not only form a clique but it is also
true that they are connected to the same set of events. Hence we can replace
them by one vertex of weight 
$$\sum_{(u_1,u_2)\in E(\overline{G[H]}(V_i\cup V_j))}w(u_1)w(u_2)=r_{ij}$$
without changing the weighted independence polynomial. But then the obtained
weighted independence polynomial is
$$I((L_H,\underline{r_e}),t)=F_{H}(\underline{r_e},t)>0$$
for $t\in [0,1]$. Then by the Scott-Sokal theorem we have
$$\mathbb{P}(\cap_{f\in E(\overline{G[H]|H})} \overline{A_f})\geq
F((H,\underline{r_e}),1)>0.$$
\end{proof}

\begin{cor} Let $\Delta$ be the largest degree of the graph $H$ and $t(H)$ be
  the largest root of the matching polynomial. Then for the
  critical edge density $d_{crit}(H)$ we have 
$$d_{crit}(H)\leq 1-\frac{1}{t(H)^2}.$$
In particular,
$$d_{crit}(H)<1-\frac{1}{4(\Delta-1)}.$$
\end{cor}

\begin{proof} Let $\gamma_e=1-r$ for every edge $e\in E(H)$, where
  $r<\frac{1}{t(H)^2}$ then 
$$F_H(\underline{r},t)=\sum_{k=0}^n(-1)^km_k(H)r^kt^k=(rt)^{n/2}M(H,\frac{1}{\sqrt{rt}})>(rt)^{n/2}M(H,t(H))=0$$ 
for $t\in [0,1]$.  Hence the set of densities $\{\gamma_e\}$ ensures the
existence of the graph $H$. Thus $d_{crit}(H)\leq 1-r$ for every
$r<\frac{1}{t(H)^2}$. Hence 
$$d_{crit}(H)\leq 1-\frac{1}{t(H)^2}.$$

The second claim follows from the fact that $t(H)< 2\sqrt{\Delta-1}$, see \cite{lieb}.
\end{proof}

%%%%%%%%%%%%%%%%%%%%%%%%%%%%%%%%%%%%%%%%%%%%%%%%%%%%%%%%%%%%%%%%%%%%%%%%%%%%%%%%%%%%%%%%%%%%%%%%%%%%%%%%%%%%%%%%%%%%
\section{Star decomposition}

In this section we examine a large class of blow-up graphs which do not
induce a given graph as a transversal. Assume that $H=H_1\cup \{v_n\}$ and we
have a blow-up graph of $H_1$ which does not induce $H_1$ as a transversal.
We can construct a blow-up graph of $H$ not inducing $H$ as follows. Let
$A_n=\{w_n\}$ be the blow-up of $v_n$. Furthermore, assume that
$N_H(v_n)=\{v_1,v_2,\dots ,v_k\}$ with the corresponding clusters
$A'_1,\dots ,A'_k$ in the blow-up of $H_1$. Then let $A_i=A'_i\cup \{w_i\}$
if $1\leq i\leq k$ and we leave unchanged all other clusters. Let us connect
$w_n$ to each elements of $A'_i$ $(1\leq k\leq n)$ and connect $w_i$ with
every possible neighbor except $w_n$. All other pairs of vertices remain
adjacent or non-adjacent as in the blow-up of $H_1$. 

Now it is clear why we call this construction a star decomposition: 
the complement of the construction with respect to $G[H]$ consists of stars,
see Figure 4.

\begin{figure}[h!]
\begin{center}
\scalebox{.55}{\includegraphics{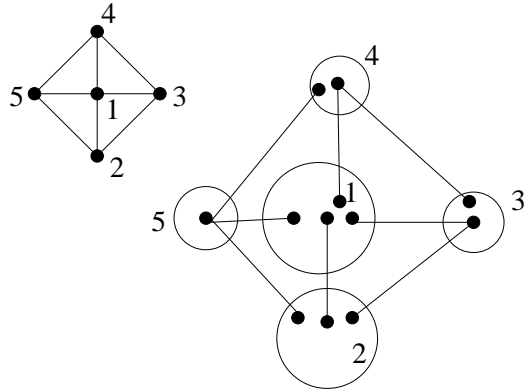}} \caption{Star decomposition
  of the wheel, the complement of the construction.} 
\end{center}
\end{figure}

This new blow-up graph will clearly not induce $H$ as a transversal.

Although we gave a construction of a blow-up of the graph $H$ not inducing
$H$, this is only the half of a real construction since we can vary the weights
of the vertices of the blow-up graph. Of course, we would like to choose the
weights optimally. But what does this mean? Assume that we are given densities
for all edges of $H$ and we wish to make a construction iteratively as we
described in the previous paragraph and now we would like to choose the
weights so that the edge-densities are at least as large as the required
edge-densities. To quantify this argument we need some definitions.
\bigskip

\begin{defn} \label{SD1} A proper labeling of the vertices of the graph $H$ is
  a bijective function $f$ from $\{1,2,\dots ,n\}$ to the set of vertices such
  that the vertex set $\{f(1),\dots ,f(k)\}$ induces a connected subgraph of
  $H$ for all $1\leq k\leq n$.
\end{defn}

\begin{defn} \label{SD2} Given a weighted graph $H$ with a proper labeling $f$,
  where  the weights on the edges are between $0$ and $1$. The \textit{weighted
    monotone-path tree} of $H$ is defined as follows. The vertices of this
  graph  are the paths of the form $f(i_1)f(i_2)\dots f(i_k)$ where
  $1=i_1<i_2<\dots <i_k$ and two such paths are connected if one is the
  extension of the other with exactly one new vertex. The weight of the edge
  connecting 
  $f(i_1)f(i_2)\dots f(i_{k-1})$ and $f(i_1)f(i_2)\dots f(i_k)$ is the weight
  of the edge $f(i_{k-1})f(i_k)$ in the graph $H$.

The monotone-path tree is the same without weights.
\end{defn} 

 \begin{figure}[h!] \label{mpt2}
\begin{center}
\scalebox{.65}{\includegraphics{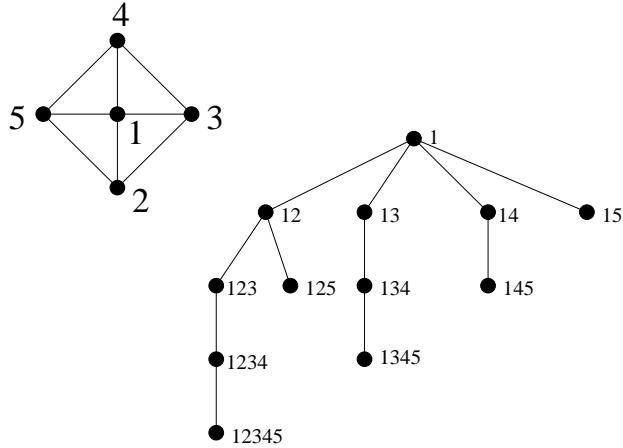}}\caption{A
  monotone-path tree of the wheel on $5$ vertices.}   
\end{center}
\end{figure}

\begin{theorem} \label{ST1} Let $H$ be  a properly labeled graph with edge
  densities $\gamma_e$ and 
  let $T_f(H)$ be its weighted monotone-path tree with weights
  $\gamma_e$. Assume that   these densities do not ensure the existence of the
  factor $T_f(H)$. Then there is a construction of a blow-up graph of $H$ not
  inducing $H$ as a transversal and all  densities between the clusters are at
  least as large as the given densities.
\end{theorem}

\begin{remark} So this theorem provides a necessary condition for the densities
  ensuring the existence of factor $H$. In fact, this gives as many necessary
  conditions as many proper labelings the graph $H$ has. The advantage of this
  theorem is that we already know the case of trees substantially.
\end{remark}

\begin{proof} We prove the statement by induction on the number of vertices of
  $H$. For $n=1,2$ the claim is trivial since $H=T_f(H)$. Now assume that we
  already know the statement till $n-1$ and we need to prove it for $|V(H)|=n$. 

  We know from Theorem~\ref{TT1} that $\gamma_e$ ensure the existence of factor
  $T=T_f(H)$ if the corresponding $\gamma'_e$  ensure the existence of factor
  $T'$. Let us apply this theorem as follows. We delete all vertices
  (monotone-paths) of $T_f(H)$ which contains the vertex $f(n)$. The remaining
  tree will be a weighted path tree of $H_1=H-\{f(n)\}$ where the new labeling
  is simply the restriction of $f$ to the set $\{1,2,\dots ,n-1\}$. (We will
  denote this restriction by $f$ as well.) By  induction
  there exists a blow-up graph of $H_1$ not inducing $H_1$ as a transversal
  and all densities between the clusters are at least $\gamma_e(T_f(H_1))$
  where we can also assume that the total weight of each cluster is $1$. 

Now we can do the the construction described in the beginning of this section.  
Let $f(n)=u$ and $N_H(u)=\{u_1,\dots ,u_k\}$. Let the weight of the new vertex
 $w_i\in A_i$ be $(1-\gamma_{uu_i})$ and the weights of the other vertices of the
 cluster be $\gamma_{uu_i}$ times the original one. Clearly, between the
 clusters $A_n$ and $A_i$ $(1\leq i\leq k)$, the weight  is just
 $\gamma_{uu_i}$ as  required. 
What about the other densities? First of all let us examine $\gamma'_e$'s.
Let us consider the adjacent vertices $f(1)\dots f(i)$ and $f(1)\dots
f(i)f(j)$ of $T_f(H_1)$. If both $f(i),f(j)\in N_H(u)$ then we deleted the
vertices $f(1)\dots f(i)f(n)$ and $f(1)\dots f(i)f(j)f(n)$ from $T_f(H)$
changing $\gamma_e=1-r_e$ to
$1-\frac{r_e}{\gamma_{f(n)f(i)}\gamma_{f(n)f(j)}}$. If 
only one of the vertices  $f(i)$ or $f(j)$ was connected to $f(n)$ then we
can still easily follow the change:
$\gamma'_e=1-\frac{r_e}{\gamma_{f(n)f(i)}}$ if $f(i)$ was connected to
$f(n)$. If none of them was connected to $f(n)$ then there is no 
change. But in all cases we do exactly the inverse of this operation at the
blow-up graphs ensuring that the new densities are at least $\gamma_e$.
\end{proof}

\begin{cor} \label{SCT1} Let $S(H)$ be the set of proper labelings of the
  graph $H$. The critical density of the graph $H$ is at least  
$$\max_{f\in S(H)}\left\{1-\frac{1}{\lambda(T_f(H))^2}\right\}.$$ 
\end{cor}

\begin{remark} If each edge density is equal to
  $1-\frac{1}{\lambda(T_f(H))^2}$ then there is a straightforward connection
  between the weights of the constructed blow-up graph and the eigenvector of
  the tree $T_f(H)$ belonging to the eigenvalue $\lambda(T_f(H))$. This
  connection is very similar to the one given by Andr\'as G\'acs.
\end{remark}

%%%%%%%%%%%%%%%%%%%%%%%%%%%%%%%%%%%%%%%%%%%%%%%%%%%%%%%%%%%%%%%%%%%%%%%%%%%%%%%%%%%%%%%%%%%%%%%%%%%%%%%%%%%%%%%%%%%%
\subsection{The Main conjecture and  a counterexample}\ \ 
\bigskip

The following conjecture seems a natural one after the case of trees.

\begin{conj}[General Star Decomposition Conjecture] \label{SC1} Let $H$ be  a
  graph with edge densities 
  $\gamma_e$. Assume that for each proper labeling $f$ the weights 
 as densities  of the weighted monotone-path tree ensure the existence of the
 graph $T_f(H)$. Then the given densities ensure the existence of the graph $H$.
\end{conj}

The following conjecture states that the bound on the critical edge density
coming from Corollary \ref{SCT1} is sharp.

\begin{conj}[Uniform Star Decomposition Conjecture] \label{SC2}
 Let $S(H)$ be the set of proper labelings of the
  graph $H$.  The critical density of the graph $H$ satisfies 
$$d_{crit}=\max_{f\in S(H)}\left\{1-\frac{1}{\lambda(T_f(H))^2}\right\}.$$ 
\end{conj}

\begin{remark} So the General Star Decomposition Conjecture asserts that for
  every graph and every weighting (or edge densities) the best we can do is to
  choose a good order of the vertices and construct the ``stars''. The Uniform
  Star Decomposition Conjecture is clearly a special case of this conjecture
  when all edge densities are the same for every edge. 
\end{remark}

The General Star Decomposition Conjecture is true for the triangle in the
sense that for every weighting the star decomposition of a suitable labeling
gives the best construction or shows that there is no suitable blow-up graph;
this is a theorem of  Adrian Bondy, Jian Shen, St\'ephan Thomass\'e and Carsten
Thomassen, see Lemma \ref{triangle1} or \cite{bond}. As we have seen this
conjecture is also true for trees. We  show that it holds also for cycles.

\begin{theorem} \label{circle} 
General Star Decomposition Conjecture holds for $C_n$.
\end{theorem}

We only sketch the proof here, since the inhomogeneous condition of edge
densities of $C_n$ is needed here, which may build up in a similar way to the
homogeneous case, described in the $4$th section of \cite{nagy1}. The details
will be left to the Reader. 

\begin{proof}[Sketch of the proof.]
Notice that a key statement in the proof of $d(C_n)=d(P_{n+1})$ \cite{nagy1}
was to make a correspondence between the constructions for $C_n$ and
$P_{n+1}$. In our terminology, this correspondence is exactly the one between
the cycle and its monotone-path tree, which is in fact a path on $n+1$ vertices.
\medskip 
 
Hence the proof can be build up, as follows.
First, by  applying Theorem \ref{ZolTh1} we can assume that each cluster has
size at most $2$. Then it  turns out that just like in the homogeneous case
\cite{nagy1} there is only one candidate for the edge-construction to give the
best construction with appropriate weighting. In fact, this edge construction
is exactly Construction $4.1$ in \cite{nagy1}.
 
Then slightly modifying Lemma $4.5$ in \cite{nagy1} we can obtain that one may
assume that one of the clusters has cardinality $1$, which provides the
correspondence of the construction of cycles and paths. 
In this case we have $n$ different paths depending on the starting vertex, and
these are exactly the monotone-path trees of the cycle. 
\end{proof}

However, in the following part we will show that the   
 General Star Decomposition Conjecture is in general false. Thus it seems
 very unlikely that the Uniform Star Decomposition Conjecture is true. Still it
 is a meaningful question to ask for which graphs one or both conjectures
 hold. The authors strongly believe that the Uniform Star Decomposition
Conjecture is true for complete graphs and complete bipartite graphs. 

\medskip

Our counterexample for the General Star Decomposition Conjecture is a weighted
bow-tie given by the following figure. 
It is not a star
decomposition in the sense we constructed it, while it is indeed a good
construction: whatever we choose from the middle 
cluster we cannot choose its neighbors (since it is the complement), but then
we have to choose the other vertices from the corresponding clusters, but they
are connected in the complement.\\

 \begin{figure}[h!] \label{bow-tie}
\begin{center}
\scalebox{.65}{\includegraphics{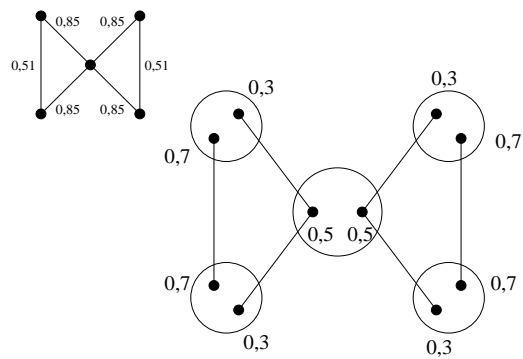}}\caption{Weighted bow-tie
  and its weighted blow-up graph of the complement.}   
\end{center}
\end{figure}

We will show that the given construction of the blow-up graph is the best
possible in the following sense. If for some blow-up graph the edge densities
are at least as large as the required densities and one of them is strictly
greater, then it induces the bow-tie as a transversal. We will also show that
no star decomposition can attain the same densities. 

Before we prove it we need some preparation. We prove a lemma which can be
considered as a generalization of Theorem~\ref{TT1}. 

\begin{lemma} \label{dens_glue} Let $H_1,H_2$ be two graphs and let $u_1\in
  V(H_1)$ and $u_2\in V(H_2)$.  
Let us denote by $H_1:H_2$ the graph obtained by identifying the vertices
$u_1,u_2$ in $H_1\cup H_2$. Let $0<m_1,m_2<1$ such that $m_1+m_2\leq 1$. 
Furthermore, assume that an  edge density $\gamma_e=1-r_e$ is assigned to
every edge. If the edge densities
$$\gamma'_e=\left\{ {\gamma_e=1-r_e \mbox{ if $e\in E(H_1)$ is not incident to
      $u_1$,}} 
    \atop    { 1-\frac{r_e}{m_1} \mbox{ if $e\in E(H_1)$ is incident
      to $u_1$,}} \right.$$
ensure the existence of a transversal $H_1$ and the edge densities
$$\gamma'_e=\left\{ {\gamma_e=1-r_e \mbox{ if $e\in E(H_2)$ is not incident to
      $u_2$,}} 
    \atop    { 1-\frac{r_e}{m_2} \mbox{ if $e\in E(H_2)$ is incident
      to $u_2$.} }\right.$$
ensure the existence of a transversal $H_2$,
then the edge densities $\{\gamma_e\}$ ensure the existence of a transversal
$H_1:H_2$.
\end{lemma}

\begin{proof} Let $G[H_1:H_2]$ be a weighted blow-up graph of $H_1:H_2$ with
  edge density $\{\gamma_e\}$. Let
$$R_1=\left\{ v\in A_{u_1=u_2}\ |\  \mbox{ $v$ can be extended to a
    transversal  $H_1\subset G[H_1]$}\right\} $$
and 
$$R_2=\left\{ v\in A_{u_1=u_2}\ |\  \mbox{ $v$ can be extended to a
    transversal  $H_2\subset G[H_2]$}\right\} .$$
We show that
$$\sum_{v\in R_1}w(v)> 1-m_1\ \ \mbox{and}\ \ \sum_{v\in R_2}w(v)> 1-m_2.$$
But then since $m_1+m_2\leq 1$ there would be some $v\in R_1\cap R_2$ which we
could extend to a transversal of $H_1$ and $H_2$ as well and thus we could find
a transversal $H_1:H_2$.
Naturally, it is enough to prove that $\sum_{v\in R_1}w(v)> 1-m_1$, because
of the symmetry. We prove it by contradiction. Assume that  $\sum_{v\in
  R_1}w(v)=1-t\leq 1-m_1$. Let us erase all vertices belonging to $R_1$ from
$A_{u_1=u_2}$ and let us give the weight $\frac{w(u)}{t}$ to the remaining
vertices $u\in A_{u_1=u_2}-R_1$. Then we obtained a weighted blow-up graph
$G'[H_1]$ in which every edge density is at least $\gamma'_e$ $(e\in E(H_1))$.
But then the assumption of the lemma ensures the existence of a transversal
$H_1$ which contradicts the construction of $G'[H_1]$. 
\end{proof} 

Now we are ready to prove that the above given construction is best possible.

\begin{counte} \label{bow-tie1}
 For graph $H$, let $V(H)=\{v_1,v_2,v_3,v_4,v_5\}$, and
$E(H)=\{v_1v_2,v_1v_3,v_1v_4,v_1v_5,v_2v_3,v_4v_5\}$. Furthermore, assume that 
 the edge densities of the blow-up graph $G[H]$ satisfy the following
inequalities: $\gamma_{12},\gamma_{13},\gamma_{14},\gamma_{15}\geq 0,85$,
$\gamma_{23},\gamma_{45}\geq 0,51$ and at least one of the inequalities is
strict. Then $G[H]$ contains $H$ as a transversal.
\end{counte} 

\begin{proof} We can assume by symmetry that at least one of the strict
  inequalities 
  $\gamma_{12}> 0,85$ or $\gamma_{23}>0,51$ holds. Let us apply 
  Lemma~\ref{dens_glue} with 
$H_1=H(v_1,v_2,v_3)$ and $H_2=H(v_1,v_4,v_5)$, $u_1=u_2=v_1$, densities
$\gamma_{ij}$ and $m_1=1/2-\varepsilon$, $m_2=1/2+\varepsilon$ where
$\varepsilon$ is a very small positive number chosen later.
Then 
$$\gamma'_{ij}\gamma'_{jk}+\gamma_{ik}-1=1-r'_{12}-r'_{13}-r'_{23}+r'_{ij}r'_{jk}>0$$ 
for any permutation $i,j,k$ of  $\{1,2,3\}$. Indeed, since
$0,3=\frac{0,15}{0,5}$ we have
$$1-0,3-0,3-0,49+0,3\cdot 0,49>1-0,3-0,3-0,49+0,3\cdot 0,3=0$$
and one of the $r_{ij}$'s is strictly smaller than $0,3$ or $0,49$ and so for
small enough $\varepsilon$, the expression
$1-r'_{12}-r'_{13}-r'_{23}+r'_{ij}r'_{jk}$ is positive. Hence by
Lemma~\ref{triangle1} it ensures the existence of a triangle transversal.
For the other triangle, $r'_{14}=\frac{r_{14}}{1/2+\varepsilon}<0,3$ and
similarly, $r'_{15}<0,3$ and $r_{45}\leq 0,49$. Again by Lemma~\ref{triangle1}
it ensures the existence of a triangle transversal. By Lemma~\ref{dens_glue}
we obtain that there exists a transversal $H$ in $G[H]$.
\end{proof}

\begin{prop} There is no weighted blow-up graph of the bow-tie arising from
  star decomposition which is at least as good as the weighted blow-up graph
  in the Figure 7.
\end{prop}

\begin{proof} Because of the symmetry and since we only need to consider the
  star decompositions where the labeling is proper, we only have to consider
  two star decompositions. Because of
  Statement~\ref{bow-tie1}, all edge densities must be exactly the required
  one. This makes the whole computation a routine work.
 
 \begin{figure}[h!] \label{bow-ties}
\begin{center}
\scalebox{.65}{\includegraphics{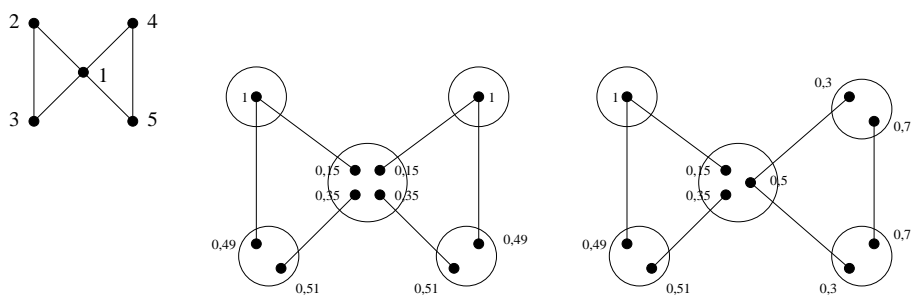}}\caption{Star decompositions of
  bow-ties.}    
\end{center}
\end{figure}

\end{proof} 

%%%%%%%%%%%%%%%%%%%%%%%%%%%%%%%%%%%%%%%%%%%%%%%%%%%%%%%%%%%%%%%%%%%%%%%%%%%%%%%%%%%%%%%%%%%%%%%%%%%%%%%%%%%%%%%%%%%%
\subsection{Complete bipartite graph case}\ \ 
\bigskip

Let $d_{crit}(K_{n,m})=d(n,m)$ be the critical edge density of the complete
bipartite graph $K_{n,m}$. Let $d_s(n,m)$ be the best edge density coming from
the star decomposition ($s$ stands for star in $d_s$). 

If one starts to do the star decomposition to $K_{n,m}$ then
the following recursion holds:
$$d_s(n,m)=\frac{1}{2-d_s(n,m-1)}\ \ \  \mbox{or}\ \ \  \frac{1}{2-d_s(n-1,m)}$$
according to which class contains the vertex $f(n+m)$. Although we have two
possibilities the recursion has only one solution, namely 
$$d_s(n,m)=1-\frac{1}{n+m-1}$$
since $d(1,1)=d_s(1,1)=0$. From this we already gain an interesting fact.

\begin{theorem} \label{BT1} For any proper labeling $f$ of the graph $K_{n,m}$
  the tree $T_f(K_{n,m})$ has spectral radius $\sqrt{n+m-1}$.  
\end{theorem}

\begin{remark} In this case a proper labeling simply means that $f(1)$ and
  $f(2)$ are elements of  different classes in the bipartite graph. \\
For different proper labelings these trees can look very different, but as
the theorem shows their spectral radii are the same. In fact, it turns out
that not only their spectral radius, but all their eigenvalues are of the form
$\pm\sqrt{n}$ where $n$ is a non-negative integer. These are the same trees
defined in the paper \cite{csiki}. 
\end{remark} 

\begin{conj} \label{BC1} $d_{crit}(K_{n,m})=d_s(n,m)=1-\frac{1}{n+m-1}.$
\end{conj}

\begin{remark}  Conjecture~\ref{SC2} clearly implies  Conjecture~\ref{BC1},
  but the authors have the feeling that Conjecture~\ref{BC1} is true while
  Conjecture~\ref{SC2} may not hold.
\end{remark}

If Conjecture \ref{BC1} holds it would have an interesting consequence.
In the case of trees and cycles the extremal construction is unique, and so it
is conjectured about the complete graphs. However, this would not stand in the
case of complete bipartite graphs; there would be several different types of
constructions depending on the proper labeling, see Figure 8.

\begin{figure}[h!]
\begin{center}
\scalebox{.75}{\includegraphics{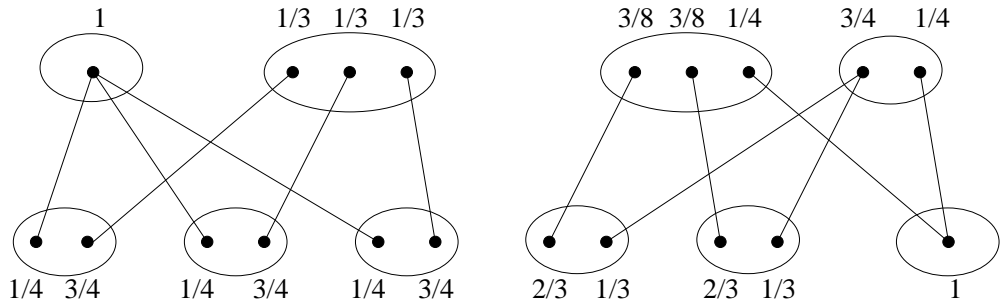}}   \caption{Two constructions
  for $G=K_{2,3}$ attaining $d_s(2,3)$}  
\end{center}
\end{figure} 
\bigskip

\noindent \textbf{Acknowledgment.} We are very grateful to the anonymous
referee for many helpful comments and remarks improving the presentation
of this paper.

%%%%%%%%%%%%%%%%%%%%%%%%%%%%%%%%%%%%%%%%%%%%%%%%%%%%%%%%%%%%%%%%%%%%%%%%%%%%%%%%%%%%%%%%%%%%%%%%%%%%%%%%%%%%%%%%%%%%

%%%%%%%%%%%%%%%%%%%%%%%%%%%%%%%%%%%%%%%%%%%%%%%%%%%%%%%%%%%%%%%%%%%%%%%%%%%%%%%%%%%%%%%%%%%%%%%%%%%%%%%%%%%%%%%%%%%%

%%%%%%%%%%%%%%%%%%%%%%%%%%%%%%%%%%%%%%%%%%%%%%%%%%%%%%%%%%%%%%%%%%%%%%%%%%%%%%%%%%%%%%%%%%%%%%%%%%%%%%%%%%%%%%%%%%%%%

\end{document}